\documentclass[12pt,oneside]{amsart}

\usepackage{amsthm,amssymb}

\usepackage{geometry}

\usepackage{xcolor}

\newtheorem{theorem}{Theorem}[section]

\newtheorem{lemma}[theorem]{Lemma}
\newtheorem{proposition}[theorem]{Proposition}

\theoremstyle{definition}

\newcommand{\PSH}{{\rm PSH}}
\newcommand{\vol}{{\rm Vol}}
\newcommand{\capa}{{\rm Cap}}

\numberwithin{equation}{section}

\usepackage{hyperref}
\hypersetup{
   pdftitle={Comparison of Monge-Amp\`ere capacities},    
   pdfauthor={Hoang-Chinh Lu},     
   colorlinks=true,       
  linkcolor=blue,          
   citecolor=blue,        
   filecolor=blue,      
   urlcolor=blue}  

\begin{document}

\title{Comparison of Monge-Amp\`ere capacities}
\author{Chinh H. Lu}
\address{Universit\'e Paris-Saclay, CNRS, Laboratoire de Math\'ematiques d'Orsay 91405 Orsay Cedex, France}
\email{\href{mailto:hoang-chinh.lu@universite-paris-saclay.fr}{hoang-chinh.lu@universite-paris-saclay.fr}}
\urladdr{\href{https://www.imo.universite-paris-saclay.fr/~lu/}{https://www.imo.universite-paris-saclay.fr/$\sim$lu/}}


\date{\today}

\begin{abstract}
Let $(X,\omega)$ be a compact K\"ahler manifold. We prove that all Monge-Amp\`ere capacities are comparable. Using this we give an alternative direct proof of the integration by parts formula for non-pluripolar products recently proved by M. Xia. 
\end{abstract}

\maketitle

\tableofcontents

\section{Introduction}
Since Yau's solution to Calabi's conjecture \cite{Yau78} geometric pluripotential theory has found its important place in the development of differential geometry. An important tool in the theory is the Monge-Amp\`ere capacity introduced by Bedford and Taylor \cite{BT82}. By analyzing capacities of sublevel sets, Ko{\l}odziej \cite{Kol98} has established a fundamental $L^{\infty}$-estimate for complex Monge-Amp\`ere equations. Several capacities have been studied in the literature with interesting applications, see \cite{GZ05,BEGZ10,DnL15,DnL17,DDL2,BT82,Kol98,DL18} and the references therein. The goal of this  note is to quantitatively compare these capacities. 

Let $(X,\omega)$ be a compact K\"ahler manifold of dimension $n$. Fix a smooth closed real $(1,1)$-form $\theta$ such that the De Rham cohomology class $\{\theta\}$ is big. Given $\psi\in\PSH(X,\theta)$ we define, for a Borel subset $E\subset X$,  
\[
{\rm Cap}_{\theta,\psi}(E)=\sup\left\{ \int_{E}\theta_{u}^{n}\ :
\ u\in\PSH(X,\theta),\ \psi-1\leq u\leq\psi\right\} .
\]
Here $\theta_{u}^{n}$ is the non-pluripolar Monge-Amp\`ere measure of $u$, see Section \ref{sect: preliminary}. 

The fact that these capacities characterize pluripolar sets suggests that they are all comparable. This is the content of our main result:
\begin{theorem}
\label{thm: Cap comparison intro} Let $\theta_{1},\theta_{2}$ be smooth closed real $(1,1)$-forms on $X$ which represent big cohomology classes. Assume that $\psi_{1}\in\PSH(X,\theta_{1})$ and $\psi_{2}\in\PSH(X,\theta_{2})$ are such that $\int_{X}(\theta_{1}+dd^{c}\psi_{1})^{n}>0$ and $\int_{X}(\theta+dd^{c}\psi_{2})^{n}>0$. Then there exist continuous functions $f,g:[0,+\infty)\rightarrow[0,+\infty)$ with $f(0)=g(0)=0$ such that, for any Borel set $E\subset X$, 
\[
\capa_{\theta_{1},\psi_{1}}(E)\leq f(\capa_{\theta_2,\psi_2}(E)),\; \capa_{\theta_2,\psi_2}(E)\leq g(\capa_{\theta_1,\psi_1}(E)).
\]
\end{theorem}

A. Trusiani has recently proved in \cite{Trus20} a comparison of Monge-Amp\`ere $\phi$-capacities for model potential $\phi$ using the metric geometry of relative finite energy classes introduced in \cite{Trus19}.

Using the comparison of capacities we provide a new proof of the integration by parts formula, a result recently proved in \cite{Xia19}. The  proof of \cite{Xia19} relies on a construction of D. Witt-Nystr\"om \cite{WN19}. Our proof uses a direct approximation method partially inspired by \cite{DNT19}.
\begin{theorem}
\cite{Xia19} \label{thm: integration by parts} Let $u,v\in L^{\infty}(X)$ be differences of quasi plurisubharmonic functions. Fix $\phi_j \in \PSH(X,\theta^j)$, $j=2,...,n$ where $\{\theta_j\}$ is big. Then
\[
\int_{X} u dd^c v \wedge \theta^2_{\phi_2}\wedge ... \wedge \theta^n_{\phi_n} =\int_{X}v dd^c u \wedge \theta^2_{\phi_2}\wedge ... \wedge \theta^n_{\phi_n}.
\]
\end{theorem}
Here, if $u=\varphi-\psi$ with $\varphi,\psi \in \PSH(X,\eta)$ then, by definition,
\[
dd^{c}u \wedge \theta^2_{\phi_2}\wedge ... \wedge \theta^n_{\phi_n}:= \eta_{\varphi} \wedge \theta^2_{\phi_2}\wedge ... \wedge \theta^n_{\phi_n} -\eta_{\psi}\wedge \theta^2_{\phi_2}\wedge ... \wedge \theta^n_{\phi_n}
\]
is a difference of non-pluripolar products, see Section \ref{sect: preliminary}.

The integration by parts formula is a key ingredient in the variational approach to solve complex Monge-Amp\`ere equations (see \cite{BBGZ13}, \cite{DDL2}). When the potentials have small unbounded locus, i.e. these are locally bounded outside a closed complete pluripolar set, the above result was proved in \cite{BEGZ10}.

The main idea of our proof of Theorem \ref{thm: integration by parts} is as follows. We first start with the simple case where $u=\varphi_1-\varphi_2$, with $\varphi_1,\varphi_2\in \PSH(X,\omega)$, vanishes in some open neighborhood of the pluripolar set $\{\varphi_1=-\infty\}$. In this case the result is a simple consequence of the plurifine locality of non-pluripolar products. For the general case we apply the first step with $\varphi_1$ and $\varphi_{2,\lambda}=\max(\varphi_1,\lambda \varphi_2)$ for $\lambda>1$. We next use the comparison of capacities above to pass to the limit as $\lambda\searrow 1$.

\subsection*{Organization of the note. }

After preparing necessary background materials in Section \ref{sect: preliminary} we systematically compare Monge-Amp\`ere capacities in Section \ref{sec: comparison of MA cap}, proving Theorem \ref{thm: Cap comparison intro}. A new proof of Theorem \ref{thm: integration by parts} is given in Section \ref{sec: integration by parts}.

\section{Preliminaries\label{sect: preliminary}}
In this section we recall necessary notions and tools in pluripotential theory. We refer the reader to \cite{BEGZ10}, \cite{BBGZ13}, \cite{DDL1,DDL2,DDL3,DDL4,DDL5} for more details. 

\subsection{Quasi plurisubharmonic functions}
Let $(X,\omega)$ be a compact K\"ahler manifold of dimension $n$. Fix a closed smooth real $(1,1)$-form $\theta$. A function $u: X\rightarrow\mathbb{R}\cup\{-\infty\}$ is quasi plurisubharmonic (qpsh) if locally $u=\rho+\varphi$ where $\rho$ is smooth and $\varphi$ is plurisubharmonic (psh). If additionally $\theta_u:=\theta+dd^c u\geq 0$ in the weak sense of currents then $u$ is called $\theta$-psh. We let $\PSH(X,\theta)$ denote the set of all $\theta$-psh functions which are not identically $-\infty$. By elementary properties of psh functions one has that $\PSH(X,\theta)\subset L^{1}(X)$. Here, if nothing is stated $L^{1}(X)$ is $L^{1}(X,\omega^{n})$. The De Rham cohomology class $\{\theta\}$ is big if $\PSH(X,\theta-\varepsilon\omega)$ is non-empty for some $\varepsilon>0$. 

Given $u,v\in{\rm PSH(X,\theta)}$ we say that $u$ is more singular than $v$, and denote by $u\preceq v$, if there exists a constant $C$ such that $u\leq v+C$ on $X$. We say that $u$ and $v$ have the same singularity type, and denote by $u\simeq v$, if $u\preceq v$ and $v\preceq u$. There is a natural least singular potential in $\PSH(X,\theta)$ given by
\[
V_{\theta}(x):=\sup\{u(x)\; :\; u\in\PSH(X,\theta),\; u\leq 0 \; \text{on}\; X\}.
\]
As is well-known $V_{\theta}$ is locally bounded in a Zariski open set called the ample locus of $\{\theta\}$. A potential $u\in\PSH(X,\theta)$ has minimal singularity type if it has the same singularity type as $V_{\theta}.$ Note that $V_{\theta}\equiv 0$ if and only if $\theta\geq 0$.

Let $\theta^1,...,\theta^n$ be smooth closed $(1,1)$-forms representing big cohomology classes. Given $u_j\in\PSH(X,\theta^j)$, $j=1,...,n$, with minimal singularity type the Monge-Amp\`ere measure
\[
(\theta^1+dd^{c}u_{1})\wedge...\wedge(\theta^n+dd^{c}u_{n})
\]
is well defined, by Bedford-Taylor \cite{BT76,BT82}, as a positive Borel measure on the intersection of the ample locus of $\{\theta^j\}$ with finite total mass. One extends this measure trivially over $X$, the resulting measure is called the non-pluripolar Monge-Amp\`ere product of $u_1,...,u_n$. For general $u_{j}\in\PSH(X,\theta^j)$ one can consider the canonical approximants $u_{j}^{t}:=\max(u_{j},V_{\theta^j}-t)$, $t>0$, $j=1,...,n$. The sequence of measures 
\[
{\bf 1}_{\cap\{u_{j}>V_{\theta^j}-t\}}(\theta^1+dd^{c}u_{1}^{t})\wedge...\wedge(\theta^n+dd^{c}u_{n}^{t})
\]
is increasing in $t$. Its strong limit, denoted by $(\theta^1+dd^{c}u_{1})\wedge...\wedge(\theta^n+dd^{c}u_{n})$, is a positive Borel measure on $X$. To simplify the notation we also denote the latter by $\theta^1_{u_1}\wedge .... \wedge \theta^n_{u_n}$. When $u_1=...=u_n$ and $\theta^1=...=\theta^n=\theta$ we obtain the non-pluripolar Monge-Amp\`ere measure of $u$, denoted by $(\theta+dd^c u)^n$ or simply by $\theta_u^n$.

We let $\mathcal{E}(X,\theta)$ denote the set of all $\theta$-psh functions $u$ with full Monge-Amp\`ere mass, i.e. such that $\int_X (\theta +dd^c u)^n =\int_X (\theta+dd^c V_{\theta})^n=\vol(\theta)$.

\subsection{Monotonicity of Monge-Amp\`ere mass}
\begin{lemma}
Assume that $u,v\in\PSH(X,\theta)$ have the same singularity type. Then $\int_{X}\theta_{u}^{n}=\int_{X}\theta_{v}^{n}$. 
\end{lemma}
The above result was first proved by D. Witt-Nystr\"om \cite{WN19}. A different proof has been recently given in \cite{LN19} using the monotonicity of the energy functional. We give below a direct proof using a standard approximation process. Another different proof has been recently given in \cite{Vu20} where generalized non-pluripolar products of positive currents are studied.

We also stress that our proof only uses the invariant of the Monge-Amp\`ere mass of bounded $\omega$-psh functions. It is thus valid on non-K\"ahler manifolds $(X,\omega)$ satisfying 
\[
\int_{X}(\omega+dd^{c}u)^{n}=\int_{X}\omega^{n},\quad\forall
u\in\PSH(X,\omega)\cap C^{\infty}(X).
\]
As shown in \cite{Chio16} the above condition is equivalent to $i\partial\bar{\partial}\omega^{k}=0$, for all $k=1,...,n-1$. 
\begin{proof}
{\bf Step 1.}  Assume that $\theta$ is a K\"ahler form.

We first prove the following claim: if there exists a constant $C>0$ such that $u=v$ on $U:=\{\min(u,v)<-C\}$ then $\int_{X}\theta_{u}^{n} = \int_{X}\theta_{v}^{n}$.

We approximate $u,v$ by $u^{t}:=\max(u,-t)$ and $v^{t}:=\max(v,-t)$. Then, for $t>C$ we have that $u^{t}=v^{t}$ on the open set $U$. This results in ${\bf 1}_{U}\theta_{u^{t}}^{n}={\bf 1}_{U}\theta_{v^{t}}^{n}$. Observe that $u^{t}=u$ on $\{u>-t\}$ and $v^{t}=v$ on $\{v>-t\}$. For $t>C$ we have $\{u\leq-t\}=\{v\leq-t\}\subset U$. Hence by plurifine locality we have 
\[
\theta_{u^{t}}^{n}={\bf 1}_{\{u>-t\}}\theta_{u}^{n}+{\bf 1}_{\{u\leq-t\}}\theta_{u^{t}}^{n}={\bf 1}_{\{u>-t\}}\theta_{u}^{n}+{\bf 1}_{\{v\leq-t\}}\theta_{v^{t}}^{n},
\]
and 
\[
\theta_{v^{t}}^{n}={\bf 1}_{\{v>-t\}}\theta_{u}^{n}+{\bf 1}_{\{v\leq-t\}}\theta_{v^{t}}^{n}={\bf 1}_{\{v>-t\}}\theta_{v}^{n}+{\bf 1}_{\{v\leq-t\}}\theta_{v^{t}}^{n},
\]
Integrating  over $X$ and noting that $\int_{X}\theta_{u^{t}}^{n}=\int_{X}\theta_{v^{t}}^{n}=\vol(\theta)$, we arrive at 
\[
\int_{\{u>-t\}}\theta_{u}^{n}=\int_{\{v>-t\}}\theta_{v}^{n}.
\]
Letting $t\to+\infty$ we prove the claim.

We now come back to the proof of the lemma in the K\"ahler case. We can assume that $v\leq u\leq v+B$, for some positive constant $B$. For each $a\in(0,1)$ we set $v_{a}:=av$ and $u_{a}:=\max(u,v_{a})$. Choosing $C>2Ba(1-a)^{-1}$ we see that $u_{a}=v_{a}$ on $\{v_{a}<-C\}$. It follows from the claim that $\int_{X}\theta_{u_{a}}^{n}=\int_{X}\theta_{v_{a}}^{n}$. By multilinearity of the non-pluripolar product we have that $\int_{X}\theta_{v_{a}}^{n}\to\int_{X}\theta_{v}^{n}$ as $a\nearrow 1$. Observe that $u_a\searrow u$  as $a\nearrow1$ hence, by \cite[Theorem 2.3]{DDL2}, 
\[
\liminf_{a\to1^{-}}\int_{X}\theta_{u_{a}}^{n}\geq\int_{X}\theta_u^n.
\]
We thus have $\int_{X}\theta_{u}^{n}\leq\int_{X}\theta_{v}^{n}$. Reversing the role of $u$ and $v$ we finally have $\int_{X}\theta_{u}^{n}=\int_{X}\theta_{v}^{n}$, finishing the proof of Step 1. 

{\bf Step 2.} We treat the general case, $\{\theta\}$ is merely big. Fix $s>0$ so large that $\theta+s\omega$ is K\"ahler. For $t>s$ we have, by the first step, 
\[
\int_{X}(\theta+t\omega+dd^{c}u)^{n}=\int_{X}(\theta+t\omega+dd^{c}v)^{n}.
\]
The multilinearity of the non-pluripolar product then gives, for all $t>s$, 
\[
\sum_{k=0}^{n}\binom{n}{k}\int_{X}\theta_{u}^{k}\wedge\omega^{n-k}t^{n-k}=\sum_{k=0}^{n}\binom{n}{k}\int_{X}\theta_{v}^{k}\wedge\omega^{n-k}t^{n-k}.
\]
We thus obtain an equality between two polynomials and identifying the coefficients we infer the desired equality.
\end{proof}

\subsection{Quasi-psh envelopes and model potentials}
Let $f=f_1-f_2$ be a difference of two quasi-psh functions.  We let $P_{\theta}(f)$ denote the largest $\theta$-psh function on $X$ lying below $f$: 
\[
P_{\theta}(f)(x) := \left( \sup \{u(x)\; :\; u \in
\PSH(X,\theta), \; u\leq f\; \text{on}\ X\}\right)^*.
\]
Here, the inequality $u\leq f$ is understood as $u+f_2\leq f_1$ on $X$. A potential $\phi\in\PSH(X,\theta)$ is called a model potential if $\int_{X}(\theta+dd^{c}\phi)^{n}>0$ and $P_{\theta}[\phi]=\phi$, where $P_{\theta}[\phi]$ is the envelope of singularity type of $\phi$, introduced by J. Ross and D. Witt-Nystr\"om \cite{RWN14}:
\[
P_{\theta}[\phi]:=\left(\lim_{t\to+\infty}P_{\theta}(\min(\phi+t,0))\right)^{*}.
\]
Given a model potential $\phi$ we let $\mathcal{E}(X,\theta,\phi)$ denote the set of $u\in\PSH(X,\theta$) more singular than $\phi$ such that $\int_{X}(\theta+dd^{c}u)^{n}=\int_{X}(\theta+dd^{c}\phi)^{n}$.

\begin{lemma}
\label{lem: envelope two forms}
If $u\in \mathcal{E}(X,\omega)$ then $P_{\theta}(u) \in \mathcal{E}(X,\theta)$.
\end{lemma}

\begin{proof}
Since $u\in \mathcal{E}(X,A\omega)$ for any $A\geq 1$, we can assume that $\omega\geq \theta$.

We first claim that, for all $b\geq 1$, $P_{\theta}(bu+(1-b)V_{\theta})\in \PSH(X,\theta)$. Indeed, set $u_j=\max(u,-j)$, $v_j:= P_{\theta}(bu_j+ (1-b)V_{\theta})$, and
\[
D:= \{v_j = bu_j+(1-b)V_{\theta}\},\; \varphi_j :=
b^{-1}v_j+(1-b^{-1})V_{\theta}.
\]
Since $\varphi_j\leq u_j$ with equality on $D$, using \cite[Lemma 4.5]{DDL5} we have
\begin{align*}
{\bf 1}_D (\theta + dd^c \varphi_j)^n \leq {\bf 1}_D (\omega +
dd^c \varphi_j)^n \leq {\bf 1}_D (\omega+dd^c u_j)^n.
\end{align*}
We choose $t>0$ so large that $b^n\int_{\{u\leq b^{-1}t\}} (\omega+dd^c u)^n < \vol(\theta)$. For $j>b^{-1}t$, by plurifine locality, we have
\begin{align*}
b^n \int_{\{u_j\leq -b^{-1}t\}} (\omega +dd^c u_j)^n &= b^n
\int_X (\omega+dd^c u_j)^n - b^n \int_{\{u >-b^{-1}t\}}
(\omega+dd^c u_j)^n\\
&=b^n \int_X (\omega+dd^c u)^n - b^n \int_{\{u >-b^{-1}t\}}
(\omega+dd^c u)^n\\
& = b^n \int_{\{u \leq -b^{-1}t\}} (\omega+dd^c u)^n <
\vol(\theta).
\end{align*}
By \cite[Lemma 4.4]{DDL5} we have that $(\theta+dd^c v_j)^n$ is supported on $D$, hence
\begin{align*}
\int_{\{v_j\leq -t\}} (\theta+dd^c v_j)^n & = \int_{\{v_j\leq
-t\} \cap D} (\theta+dd^c v_j)^n \leq b^n \int_{\{v_j\leq -t\} \cap D} (\theta +dd^c
\varphi_j)^n \\
 & \leq b^n \int_{\{u_j\leq -b^{-1}t\}}  (\omega +dd^c u_j)^n < \vol(\theta).
\end{align*}
It thus follows that $\sup_X v_j > -t$, hence $v_j \searrow v \in \PSH(X,\theta)$, proving the claim.

Observe also that $P_{\theta}(u) \geq b^{-1} P_{\theta}(bu+(1-b)V_{\theta}) +(1-b^{-1})V_{\theta}$. It thus follows from \cite{WN19} and muntilinearity of the non-pluripolar product that
\[
\int_X (\theta +dd^c P_{\theta}(u))^n \geq (1-b^{-1})^n
\vol(\theta) \to \vol(\theta)
\]
as $b\to +\infty$. This proves that $P_{\theta}(u) \in \mathcal{E}(X,\theta)$.
\end{proof}

\subsection{Monge-Amp\`ere capacities}
Fix a $\theta$-psh function $\psi \leq 0$. We define, for each Borel set $E\subset X$, 
\[
\capa_{\theta,\psi}(E):=\sup\left\{ \int_{E}\theta_{u}^{n}\ :\
u\in\PSH(X,\theta),\ \psi-1\leq u\leq\psi\right\} .
\]
Given a Borel subset $E$, the global $\phi$-extremal function is defined by 
\[
V_{E,\theta,\phi}(x):=\sup\{v(x) \; : \; v\in\PSH(X,\theta),\;
v\preceq\phi,\ v\leq\phi\ \text{on}\ E\}, x\in X.
\]
It was shown in \cite{DDL2,DDL4}, when $\phi$ is a model potential and $E$ is non pluripolar, that $V_{E,\theta,\phi}^{*}$ is a $\theta$-psh function having the same singularity type as $\phi$. Moreover $V_{E,\theta,\phi}^{*}=\phi$ on $E$ modulo a pluripolar set. We set $M_{E,\theta,\phi}:=\sup_{X}V_{E,\theta,\phi}^*$. In case when $\phi=V_{\theta}$ we will simplify the notation by setting $V_{E,\theta}:=V_{E,\theta,\phi}$ and $M_{E,\theta}:=M_{E,\theta,\phi}$.
\begin{lemma}
\label{lem: modulo pluripolar} 
 Let $\phi\in\PSH(X,\theta)$ be such that $\int_{X}\theta_{\phi}^{n}>0$. If $E\subset X$ is a Borel se and $P\subset X$ is a pluripolar set then $V_{E\cup P,\theta,\phi}^{*}=V_{E,\theta,\phi}^{*}$.
\end{lemma}

\begin{proof}
It follows from the definition that $V_{E,\theta,\phi}\geq V_{E\cup P,\theta,\phi}$ since $E\subset E\cup P$. Let now $u\in\PSH(X,\theta)$ be a candidate defining $V_{E,\theta,\phi}$. We claim that there exists $v\in \PSH(X,\theta)$ such that $v\leq\phi$ and $P\subset\{v=-\infty\}$. Indeed, it follows from \cite{GZ05,GZ07} that there exists $v_0\in \mathcal{E}(X,\omega)$ such that $P\subset \{v_0=-\infty\}$. By Lemma \ref{lem: envelope two forms} we have $P_{\theta}(v_0) \in \mathcal{E}(X,\theta)$, hence \cite[Lemma 5.1]{DDL5} ensures that $v:=P_{\theta}(\min(\phi,v_0)) \in \PSH(X,\theta)$. Note also that $P\subset \{v=-\infty\}$ and $v\leq \phi$. This proves the claim.

For each $\lambda\in(0,1)$ the function $u_{\lambda}:=\lambda v+(1-\lambda)u$ is $\theta$-psh and satisfies $u_{\lambda}\preceq\phi$, $u_{\lambda}\leq\phi$ on $E\cup P$. We thus have $u_{\lambda}\leq V_{E\cup P,\theta,\phi}^{*}$. Letting $\lambda\to0^{+}$ we obtain $u\leq V_{E\cup P,\theta,\phi}^*$ on $X$ modulo a pluripolar set, hence on the whole $X$. This finally gives $V_{E,\theta,\phi}^{*}\leq V_{E\cup P,\theta,\phi}^{*}$. 
\end{proof}
\begin{proposition}
If $\psi\in\PSH(X,\theta)$ satisfies $\int_{X}(\theta+dd^{c}\psi)^{n}>0$ then $\capa_{\theta,\psi}$ characterizes pluripolar sets: for all Borel sets $E$ we have 
\[
\capa_{\theta,\psi}(E)=0\Longleftrightarrow E\;\text{is
pluripolar. }
\]
\end{proposition}

The proof below is quasi identical to that of \cite[Lemma 4.3]{DDL2}.
\begin{proof}
If $E$ is pluripolar then by definition $\capa_{\theta,\psi}(E)=0$. Conversely, assume that $E$ is non pluripolar. Then there exists a compact set $K\subset E$ such that $K$ is non pluripolar. 

Let $V_{K,\theta}$ be the global extremal $\theta$-psh function of $K$. Then $V_{K,\theta}^{*}\in\PSH(X,\theta)$ has minimal singularity type. For $t>0$ we set 
\[
u_{t}:=P_{\theta}(\min(\psi+t,V_{K}^{*})).
\]
It is well known that $\theta_{V_{K,\theta}^*}^{n}$ is supported on $K$. By \cite[Lemma 3.7]{DDL2} 
\[
0<\int_{X}\theta_{\psi}^{n}=\int_{X}\theta_{u_{t}}^{n}\leq\int_{\{u_{t}=\psi+t\}}\theta_{\psi}^{n}+\int_{K}\theta_{u_{t}}^{n}.
\]
The first term on the right-hand side converges to $0$ as $t\to+\infty$. Thus for $t>1$ large enough we have $\int_{K}\theta_{u_{t}}^{n}>0$, hence $\capa_{\theta,\psi}(K)>0$. 
\end{proof}

A sequence of functions $u_j$ converges in capacity to $u$ if, for any $\varepsilon>0$,
\[
\lim_{j\to +\infty} \capa_{\omega}(\{x\in X\; :\;
|u_j(x)-u(x)|>\varepsilon\}) =0.
\]
We will also need the following convergence result whose proof is quasi identical to the proof of \cite[Corollary 2.9]{DDL1}:
\begin{theorem}
\label{thm: convergence}
Assume that $\mu_j$ is a sequence of positive Borel measures converging weakly to $\mu$. Assume that there exists a continuous function $f: [0,+\infty) \rightarrow [0,+\infty)$ with $f(0)=0$ such that, for any Borel set $E$,
\[
\mu_j (E) + \mu(E) \leq f \left(\capa_{\omega} (E)\right).
\]
Let $u_j$ be a sequence of uniformly bounded quasi-continuous function which converges in capacity to $u$. Then $u_j \mu_j \rightarrow u \mu$ in the sense of measures on $X$. 
\end{theorem}

\begin{proof}
Fixing $\varepsilon>0$ there exist $v,v_j$ continuous functions on $X$ such that
\[
\capa_{\omega}(\{x\in X\; : \; u_j(x) \neq v_j(x) \; \text{or}\;
u(x)\neq v(x)\}) < \varepsilon.
\]
Let $A>0$ be a constant such that $|u_j|+|v_j|+|u|+|v|\leq A$ on $X$. Fix $\delta>0$. For $j>N$ large enough we have, by the assumption that $u_j$ converges in capacity to $u$, that
\[
\capa_{\omega} (\{x\in X\; : \; |u_j(x) -u(x)|>\delta ) <
\varepsilon.
\]
Fixing a continuous function $\chi$, it follows from the above that
\begin{align*}
&|\int_X (\chi u_j \mu_j - \chi u d\mu) | \leq \int_X |\chi
(u_j-u) | \mu_j + |\int_X \chi u (\mu_j -\mu) | \\
&\leq \delta \int_X |\chi| \mu_j + C\sup_X |\chi| O(\varepsilon)
+ |\int_X \chi (u-v) (\mu_j-\mu)| + |\int_X \chi v
(\mu_j-\mu)|\\
&\leq \delta \int_X |\chi| \mu_j + 2C\sup_X |\chi|
O(\varepsilon) + |\int_X \chi v (\mu_j-\mu)|.
\end{align*}
Since $v$ is continuous on $X$ the last term converges to $0$ as $j\to +\infty$. This completes the proof.
\end{proof}

\section{Comparison of Monge-Amp\`ere capacities\label{sec: comparison of MA cap}}
In this section we establish a comparison between Monge-Amp\`ere capacities. We first prove a version of the Chern-Levine-Nirenberg inequality.
\begin{lemma}
\label{lem: CLN} Assume that $u,v\in\PSH(X,\omega)$ have the same singularity type, $v\leq u\leq v+B$, and $\psi$ is a bounded $\omega$-psh function. Then 
\[
\int_{X}\psi\omega_{u}^{n}\geq\int_{X}\psi\omega_{v}^{n}-nB\int_{X}\omega^{n}.
\]
\end{lemma}

\begin{proof}
We can assume that $u\geq v$. We first prove the lemma under the assumption that $u=v$ on the open set $U:=\{v<-C\}$, for some positive constant $C$.

We approximate $u$ and $v$ by $u^{t}:=\max(u,-t)$ and $v^{t}:=\max(v,-t)$. For $t>0$ we apply the integration by parts formula for bounded $\omega$-psh functions, which is a consequence of Stokes theorem, to get 
\[
\int_{X}\psi(\omega_{u^{t}}^{n}-\omega_{v^{t}}^{n})=\int_{X}(u^{t}-v^{t})dd^{c}\psi\wedge
S^{t},
\]
where $S^{t}:=\sum_{k=0}^{n-1}\omega_{u^{t}}^{k}\wedge\omega_{v^{t}}^{n-1-k}$. Since $u^{t}\geq v^{t}$ we can continue the above estimate and obtain
\[
\int_{X}\psi(\omega_{u^{t}}^{n}-\omega_{v^{t}}^{n})=\int_{\Omega}(u^{t}-v^{t})(\omega_{\psi}\wedge
S^{t}-\omega\wedge S^{t})\geq-Bn\int_{X}\omega^{n}.
\]
For $t>B+C$ we have that $u^{t}=v^{t}$ on the open set $U$ which contains $\{u\leq-t\}=\{v\leq-t\}$. It thus follows that ${\bf 1}_{U}\omega_{u^{t}}^{n}={\bf 1}_{U}\omega_{v^{t}}^{n}$. Thus, for $t>B+C$ we have 
\begin{eqnarray*}
\int_{\{v>-t\}}\psi(\omega_{u}^{n}-\omega_{v}^{n}) & = &
\int_{X}\psi(\omega_{u^{t}}^{n}-\omega_{v^{t}}^{n})\geq-Bn\int_{X}\omega^{n}.
\end{eqnarray*}
Letting $t\to+\infty$ we prove the claim.

We come back to the proof of the lemma. By approximating $\psi$ from above by smooth $\omega$-psh functions, see \cite{Dem94}, \cite{BK07}, we can assume that $\psi$ is smooth. We fix $a\in(0,1)$ and set $v_{a}:=av$, $u_{a}:=\max(u,v_{a})$. Then for some constant $C>0$ large enough we have that $u_{a}=v_{a}$ on $\{v_{a}<-C\}$. We can thus apply the first step to get  
\[
\int_{X}\psi\omega_{u_{a}}^{n}\geq\int_{X}\psi\omega_{v_{a}}^{n}-nB\int_{X}\omega^{n}.
\]
Letting $a\nearrow 1$ and using \cite{DDL2} we obtain 
\[
\int\psi\omega_{u}^{n}\geq\int_{X}\psi\omega_{v}^{n}-B'.
\]
\end{proof}
\begin{lemma}
\label{lem: cap upper bound} 
Let $\phi\in\PSH(X,\omega)$ be suchthat $P_{\omega}[\phi]=\phi$ and $\int_{X}\omega_{\phi}^{n}>0$. Then for any Borel set $E\subset X$ we have 
\[
\capa_{\omega,\phi}(E)\leq\frac{C}{M_{E,\omega,\phi}},
\]
where $C>0$ is a uniform constant independent of $\phi$. 
\end{lemma}

Note that the above estimate holds for a big class $\{\theta\}$ as well but to prove this we need to invoke the integration by parts formula in Section \ref{sec: integration by parts}.
\begin{proof}
Fix $C_{0}$ a positive constant such that for all $v\in\PSH(X,\omega)$ with $\sup_{X}v=0$ we have $\int_{X}|v|\omega^{n}\leq C_{0}$. The existence of $C_{0}$ follows from \cite[Proposition 2.7]{GZ05}.

We can assume that $0<M_{E,\omega,\phi}<+\infty$. Let $u\in\PSH(X,\omega)$ be such that $\phi-1\leq u\leq\phi$. Observe that the function $V_{E,\omega,\phi}^{*}-M_{E,\omega,\phi}$ is $\omega$-psh satisfying $\sup_{X}(V_{E,\omega,\phi}^{*}-M_{E,\omega,\phi})=0$. As recalled above we thus have 
\[
\int_{X}|V_{E,\omega,\phi}^{*}-M_{E,\omega,\phi}-\phi|\omega^{n}\leq2C_{0}.
\]
We also have that $|V_{E,\omega,\phi}^{*}-M_{E,\omega,\phi}-\phi|=M_{E,\omega,\phi}$ on $E$ modulo a pluripolar set. By Lemma \ref{lem: CLN} we have that, for all negative $v\in\PSH(X,\omega)$, 
\[
\int_{X}|v|(\omega_{u}^{n}-\omega_{\phi}^{n})\leq
n\int_{X}\omega^{n}.
\]
By \cite[Theorem 3.8]{DDL2} we also have that $\omega_{\phi}^{n}\leq\omega^{n}$. We thus have, for all  $v\in\PSH(X,\omega)$ normalized by $\sup_{X}v=0$,
\[
\int_{X}|v|\omega_{u}^{n}\leq n\int_{X}\omega^{n}+C_{0}.
\]
It thus follows from Lemma \ref{lem: CLN} and the triangle inequality that 
\begin{eqnarray*}
\int_{E}\omega_{u}^{n} & \leq &
\frac{1}{M_{E,\omega,\phi}}\int_{X}|V_{E,\omega,\phi}^{*}-M_{E,\omega,\phi}-\phi|\omega_{u}^{n}\\
& \leq &
\frac{1}{M_{E,\omega,\phi}}\left(2n\int_{X}\omega^{n}+4C_{0}\right).
\end{eqnarray*}
Taking the supremum over all candidates $u$ we obtain the desired inequality. 
\end{proof}

\begin{lemma}
\label{lem: GLZ} Fix $\varphi,\psi\in{\rm PSH(X,\theta)}$ such that $\psi\leq\varphi$ and $\int_{X}\theta_{\varphi}^n=\int_{X}\theta_{\psi}^n$. Then there exists a continuous function $g:[0,+\infty)\rightarrow[0,+\infty)$ with $g(0)=0$ such that, for all Borel sets $E$, 
\[
\capa_{\theta,\psi}(E)\leq
g\left(\capa_{\theta,\varphi}(E)\right).
\]
\end{lemma}
Our proof uses an idea in \cite{GLZ19}.
\begin{proof}
We can assume that $\varphi\leq 0$. Let $\chi:(-\infty,0]\rightarrow(-\infty,0]$ be an increasing function such that $\chi(-\infty)=-\infty$ and
\[
A:=\int_{X}|\chi(\psi-1-\varphi)|\theta_{\psi}^{n}<+\infty.
\]
We claim that if $v\in\PSH(X,\theta)$ with $\varphi-t\leq v\leq\varphi$ then for any Borel set $E$ we have 
\[
\int_{E}\theta_{v}^{n}\leq\max(t,1)^{n}\capa_{\theta,\phi}(E).
\]
Indeed, for $t\geq1$, the function $v_{t}:=t^{-1}v+(1-t^{-1})\varphi$ is $\theta$-psh and $\varphi-1\leq v\leq\varphi$. We thus have 
\[
t^{-n}\int_{E}\theta_{v}^{n}\leq\int_{E}\theta_{v_{t}}^{n}\leq\capa_{\theta,\varphi}(E),
\]
yielding the claim. Let $u$ be a $\theta$-psh function such that$\psi-1\leq u\leq\psi$. Fix $t>0$ and set $u_{t}:=\max(u,\varphi-2t)$, $E_{t}:=E\cap\{u>\varphi-2t\}$, $F_{t}:=E\cap\{u\leq\varphi-2t\}$. By plurifine locality and the claim we have that 
\[
\int_{E_{t}}\theta_{u}^{n}=\int_{E_{t}}\theta_{u_{t}}^{n}\leq(2t)^{n}\capa_{\theta,\phi}(E_{t})\leq(2t)^{n}\capa_{\theta,\phi}(E).
\]
On the other hand, using the inclusions 
\[
F_{t}\subset\left\{ \psi-1\leq\frac{u+\varphi}{2}-t\right\}
\subset\{\psi-1\leq\varphi-t\}
\]
and the comparison principle \cite[Corollary 3.6]{DDL2} we infer
\[ 
\int_{F_{t}}\theta_{u}^{n}\leq2^{n}\int_{\{\psi\leq\varphi-t+1\}}\theta_{\psi}^{n}\leq\frac{2^{n}}{|\chi(-t)|}\int_{X}|\chi(\psi-1-\varphi)|\theta_{\psi}^{n}.
\]
Taking the supremum over all candidates $u$ we obtain 
\[
\capa_{\theta,\psi}(E)\leq(2t)^{n}\capa_{\theta,\phi}(E)+\frac{2^{n}A}{|\chi(-t)|}.
\]
Taking $t:=(\capa_{\theta,\varphi}(E))^{-1/2n}>1$, we get $\capa_{\theta,\psi}(E)\leq g\left(\capa_{\theta,\varphi}(E)\right)$, where $g$ is defined on $[0,+\infty)$ by 
\[
g(s):=2^{n}s^{1/2}+\frac{2^{n}A}{|\chi(-s^{-1/2n})|}.
\]
\end{proof}

\begin{lemma}
\label{lem: DDL} 
Assume that $\phi\in\PSH(X,\omega)$, $\int_{X}\omega_{\phi}^{n}>0$ and $P_{\omega}[\phi]=\phi$. Then there exists a constant $A>0$ such that for any Borel set $E$ we have 
\[
A^{-1}\capa_{\omega}(E)^{n}\leq\capa_{\omega,\phi}(E)\leq
A\;\left(\capa_{\omega}(E)\right)^{1/n}.
\]
\end{lemma}

The proof uses an idea in \cite{DDL1}. 
\begin{proof}
By inner regularity of the capacity we can assume that $E$ is compact. By Lemma \ref{lem: cap upper bound} and \cite[Lemma 4.9]{DDL2} we have 
\[
\capa_{\omega}(E)^{n}\leq CM_{E,\omega}^{-n}\leq
CM_{E,\omega,\phi}^{-n}\leq C'\capa_{\omega,\phi}(E),
\]
proving the left-hand side inequality. We next prove the right-hand side inequality. By \cite[Lemma 4.3]{DDL5} there exists a
constant $b>1$ such that $P_{\omega}(\lambda\phi)\in\PSH(X,\omega)$. Set
\[
v:=(1-b^{-1})V_{E,\omega}^{*}+b^{-1}P_{\omega}(b\phi).
\]
Recall that $V_{E,\omega}=V_{E,\omega,0}$ is the global extremal function of $E$ which takes values $0$ on $E$ modulo a pluripolar set. As$V_{E,\omega}^{*}$ is bounded we have that $v\in\PSH(X,\omega)$, $v\preceq\phi$, and $v\leq\phi$ on $E$ modulo a pluripolar set. By Lemma \ref{lem: modulo pluripolar} we thus have $v\leq V_{E,\omega,\phi}^{*}$. Set $C_{0}:=-\sup_{X}P_{\omega}(b\phi)\geq 0$ and $G:=\{P_{\omega}(2\phi)\geq-C_{0}-1\}$. Note that $G$ has positive Lebesgue measure, hence $G$ is non pluripolar. In particular $M_{G,\omega}<+\infty$. We have 
\[
\sup_{X}V_{E,\omega,\phi}^{*}\geq\sup_{X}v\geq\sup_{G}v\geq(1-b^{-1})\sup_{G}V_{E,\omega}^{*}-C_{1}.
\]
On the other hand we have that $u:=V_{E,\omega}^{*}-\sup_{G}V_{E,\omega}^{*}$ is $\omega$-psh and $u\leq0$ on $G$. It thus follows that $u\leq M_{G,\omega}<+\infty$, hence $\sup_{G}V_{E,\omega}^{*}\geq V_{E,\omega}^{*}-M_{G,\omega}$. Taking the supremum over $X$ we get $\sup_{G}V_{E,\omega}^{*}\geq M_{E,\omega}-M_{G,\omega}$. Therefore
\[
M_{E,\omega,\phi}\geq(1-b^{-1})M_{E,\omega}-C_{2}.
\]
It follows from \cite[Proposition 7.1]{GZ05} that $\capa_{\omega}(E)\geq C_{3}M_{E,\omega}^{-n}$, for some uniform constant $C_{3}>0$. Set $a=C_{3}^{-1}(2b(b-1)^{-1}C_{2})^{-n}$. If $\capa_{\omega}(E)\leq a$ then 
\[
(1-b^{-1})M_{E,\omega}\geq(1-b^{-1})(aC_{3})^{-1/n}=2C_{2}.
\]
From the above we thus have $M_{E,\omega,\phi}\geq C_{5}M_{E,\omega}$. Then by Lemma \ref{lem: cap upper bound} and \cite[Proposition 7.1]{GZ05} we have 
\[
\capa_{\omega,\phi}(E)\leq C_{6}\capa_{\omega}(E)^{1/n}.
\]
Observe that $\capa_{\omega,\phi}(E)\leq\int_{X}\omega^{n}$. Let$C_{7}$ be a positive constant such that $C_{7}\geq C_{6}$ and $C_{7}a^{1/n}\geq\int_{X}\omega^{n}$. We then have 
\[
\capa_{\omega,\phi}(E)\leq C_{7}\capa_{\omega}(E)^{1/n}.
\]
\end{proof}
The main result of this note is a direct consequence of the following:
\begin{theorem}
\label{thm: Cap comparison} 
Assume that $\psi\in \PSH(X,\theta)$ and $\int_{X}\theta_{\psi}^{n}>0$. Then there exist continuous functions $f,g:[0,+\infty)\rightarrow[0,+\infty)$ with $f(0)=g(0)=0$ such that, for any Borel set $E$, 
\[
\capa_{\theta,\psi}(E)\leq
f\left(\capa_{\omega}(E)\right)\quad\text{and}\quad\capa_{\omega}(E)\leq
g\left(\capa_{\theta,\psi}(E)\right).
\]
\end{theorem}

\begin{proof}
By inner regularity of the capacities we can assume that $E$ is compact. By scaling we can assume that $\theta\leq\omega$. Set $\phi:=P_{\omega}[\psi]$. It follows from Lemma \ref{lem: GLZ} that 
\[
\capa_{\omega,\psi}\leq f\left(\capa_{\omega,\phi}\right),
\]
for some continuous function $f$ with $f(0)=0$, while Lemma \ref{lem: DDL} gives 
\[
\capa_{\omega,\phi}\leq A\;\capa_{\omega}^{1/n},
\]
for some positive constant $A$. Combining these two inequalities we obtain the first inequality of the theorem. We next prove the second one. By \cite[Proposition 7.1]{GZ05} and \cite[Lemma 4.9]{DDL2} we have 
\[
\capa_{\omega}(E)\leq CM_{E,\omega}^{-1}\leq
CM_{E,\theta,\phi}^{-1}\leq C'\capa_{\theta,\phi}(E)^{1/n}.
\]
Since $\int_{X}(\theta+dd^{c}\psi)^{n}>0$, by \cite[Corollary 3.20]{LN19} $P_{\theta}(2\psi-\phi)\in{\rm \mathcal{E}(X,\theta,\phi)}$.  Setting
$u:=\frac{P_{\theta}(2\psi-\phi)+\phi}{2}\leq\psi$, by Lemma
\ref{lem: GLZ}
we have $\capa_{\theta,u}\leq g(\capa_{\theta,\psi})$, for some
continuous
function $g$ with $g(0)=0$. The proof is finished if we can show that $\capa_{\theta,\phi}\leq2^{n}\capa_{\omega,u}.$ Take
$v\in\PSH(X,\theta)$
such that $\phi-1\leq v\leq\phi$. Then 
\[
u-1\leq h:=\frac{v+P_{\theta}(2\psi-\phi)}{2}\leq u,
\]
 and hence 
\[
\int_{E}\theta_{v}^{n}\leq2^{n}\int_{E}\theta_{h}^{n}\leq2^{n}\capa_{\theta,u}(E).
\]
Taking the supremum over all $v$ we obtain
$\capa_{\theta,\phi}\leq2^{n}\capa_{\theta,u}$. 
\end{proof}

\section{Integration by parts\label{sec: integration by parts}}
The integration by parts formula was recently studied in \cite{Xia19} using Witt-Nystr\"om's construction. In this section we give an alternative direct proof which also applies  to the setting of complex $m$-Hessian equations considered in \cite{LN19}. We first start with the following key lemma.
\begin{lemma}
\label{lem: pre integration by parts} 
Let $\varphi_{1},\varphi_{2},\psi_{1},\psi_{2}\in\PSH(X,\theta)$ be such that $\varphi_{1}\simeq\varphi_{2}$ and $\psi_{1}\simeq\psi_{2}$. Then 
\[
\int_{X}(\varphi_{1}-\varphi_{2})\left(\theta_{\psi_{1}}^{n}-\theta_{\psi_{2}}^{n}\right)=\int_{X}(\psi_{1}-\psi_{2})(S_{1}-S_{2}),
\]
where $S_{j}:=\sum_{k=0}^{n-1}\theta_{\varphi_{j}}\wedge\theta_{\psi_{1}}^{k}\wedge\theta_{\psi_{2}}^{n-k-1}$, $j=1,2$. 
\end{lemma}

\begin{proof}
It follows from \cite[Theorem 2.4]{DDL2} that $\int_{X}(\theta_{\psi_{1}}^{n}-\theta_{\psi_{2}}^{n})= \int_{X}(S_{1}-S_{2})=0$. By adding a constant we can assume that $\varphi_{1},\varphi_{2},\psi_{1},\psi_{2}$
are negative. 

{\bf Step 1.} We assume that $\theta$ is K\"ahler and $\psi_{1},\psi_{2},\varphi_{1},\varphi_{2}$ are $\lambda\theta$-psh for some $\lambda\in(0,1)$. 

{\bf Step 1.1.} We also assume that there exists $C>0$ such that $\psi_{1}=\psi_{2}$ on the open set $U:=\{\min(\psi_{1},\psi_{2})<-C\}$ and $\varphi_{1}=\varphi_{2}$ on the open set $V:=\{\min(\varphi_{1},\varphi_{2})<-C\}$.

For a function $u$ we consider its canonical approximant $u^{t}:=\max(u,-t)$, $t>0$. It follows from Stokes theorem that 
\[
\int_{X}(\varphi_{1}^{t}-\varphi_{2}^{t})\left(\theta_{\psi_{1}^{t}}^{n}-\theta_{\psi_{2}^{t}}^{n}\right)=\int_{X}(\psi_{1}^{t}-\psi_{2}^{t})(S_{1}^{t}-S_{2}^{t}),
\]
where $S_{j}^{t}:=\sum_{k=0}^{n-1}\theta_{\varphi_{j}^{t}}\wedge\theta_{\psi_{1}^{t}}^{k}\wedge\theta_{\psi_{2}^{t}}^{n-k-1}$, $j=1,2$. Fix $t>C$. Since $\psi_{1}^{t}=\psi_{2}^{t}$ in the open set $U$ and $\{\psi_{1}\leq-t\}=\{\psi_{2}\leq-t\}\subset U$ it follows that ${\bf 1}_{\{\psi_{1}\leq-t\}}\theta_{\psi_{1}^{t}}^{n}={\bf 1}_{\{\psi_{1}\leq-t\}}\theta_{\psi_{2}^{t}}^{n}$. Moreover, by plurifine locality of the non-pluripolar product we have
\begin{eqnarray*}
\int_{X}(\varphi_{1}^{t}-\varphi_{2}^{t})\left(\theta_{\psi_{1}^{t}}^{n}-\theta_{\psi_{2}^{t}}^{n}\right)
& = &
\int_{\{\psi_{1}>-t\}}(\varphi_{1}^{t}-\varphi_{2}^{t})(\theta_{\psi_{1}^{t}}^{n}-\theta_{\psi_{2}^{t}}^{n})\\
& = &
\int_{\{\psi_{1}>-t\}}(\varphi_{1}^{t}-\varphi_{2}^{t})(\theta_{\psi_{1}}^{n}-\theta_{\psi_{2}}^{n}).
\end{eqnarray*}
Letting $t\to+\infty$ we obtain 
\[
\lim_{t\to+\infty}\int_{X}(\varphi_{1}^{t}-\varphi_{2}^{t})\left(\theta_{\psi_{1}^{t}}^{n}-\theta_{\psi_{2}^{t}}^{n}\right)=\int_{X}(\varphi_{1}-\varphi_{2})\left(\theta_{\psi_{1}}^{n}-\theta_{\psi_{2}}^{n}\right).
\]
Using the fact that $\varphi_{1}^{t}=\varphi_{2}^{t}$ on $\{\varphi_{1}\leq-t\}=\{\varphi_{2}\leq-t\}$ which is contained in the open set $V$ we have that 
\[
{\bf 1}_{\{\varphi_{1}\leq-t\}}S_{1}^{t}={\bf
1}_{\{\varphi_{1}\leq-t\}}S_{2}^{t}.
\]
We thus have 
\[
\int_{X}(\psi_{1}^{t}-\psi_{2}^{t})(S_{1}^{t}-S_{2}^{t})=\int_{\{\psi_{1}>-t\}\cap\{\varphi_{1}>-t\}}(\psi_{1}-\psi_{2})(S_{1}-S_{2}).
\]
Letting $t\to+\infty$ we finish Step 1.1.

{\bf Step 1.2.} We remove the assumptions made in Step 1.1.

It follows from  \cite[Theorem 2.4]{DDL2} that  $\int_{X}(\theta_{\psi_{1}}^{n}-\theta_{\psi_{2}}^{n})=\int_{X}(S_{1}-S_{2})=0$.  Thus adding a constant we can assume that $\varphi_{1}\leq\varphi_{2}$ and $\psi_{1}\leq\psi_{2}$. Let $B>0$ be a constant such that 
\[
\varphi_{2}\leq\varphi_{1}+B\ ;\ \psi_{2}\leq\psi_{1}+B.
\]
For each $\varepsilon\in(0,\frac{1}{\lambda}-1)$ we define 
\[
\psi_{2,\varepsilon}:=\max(\psi_{1},(1+\varepsilon)\psi_{2})\ ;\
\varphi_{2,\varepsilon}:=\max(\varphi_{1},(1+\varepsilon)\varphi_{2}).
\]
Observe that $\psi_{1}\leq\psi_{2,\varepsilon}\leq\psi_{1}+B$ and $\varphi_{1}\leq\varphi_{2,\varepsilon}\leq\varphi_{1}+B$. These are $\omega$-psh functions satisfying the assumptions in Step 1.1 with $C=B+B\varepsilon^{-1}$. Indeed, if $\varphi_{1}(x)<-C$ then
\[
(1+\varepsilon)\varphi_{2}(x)=\varphi_{2}(x)+\varepsilon\varphi_{2}(x)\leq\varphi_{1}(x)+B+\varepsilon(B-C)\leq\varphi_{1}(x).
\]
We can thus apply Step 1.1 to $\psi_{1}$, $\psi_{2,\varepsilon}$, $\varphi_{1}$,$\varphi_{2,\varepsilon}$ to obtain 
\[
\int_{X}(\varphi_{1}-\varphi_{2,\varepsilon})\left(\theta_{\psi_{1}}^{n}-\theta_{\psi_{2,\varepsilon}}^{n}\right)=\int_{X}\left(\psi_{1}-\psi_{2,\varepsilon}\right)(S_{1,\varepsilon}-S_{2,\varepsilon}),
\]
where $S_{1,\varepsilon}:=\sum_{k=0}^{n-1}\theta_{\varphi_{1}}\wedge\theta_{\psi_{1}}^{k}\wedge\theta_{\psi_{2,\varepsilon}}^{n-k-1}$ and $S_{2,\varepsilon}:=\sum_{k=0}^{n-1}\theta_{\varphi_{2,\varepsilon}}\wedge\theta_{\psi_{1}}^{k}\wedge\theta_{\psi_{2,\varepsilon}}^{n-k-1}$. By Theorem \ref{thm: Cap comparison} there exists a continuous function $f:[0,+\infty)\rightarrow[0,+\infty)$ with $f(0)=0$ such that for every Borel set $E$, 
\[
\capa_{\theta,\psi}(E)\leq f(\capa_{\theta}(E)),
\]
where $\psi:=\frac{\varphi_{1}+\varphi_{2}+\psi_{1}+\psi_{2}}{5}-B$ is a $\theta$-psh function with $\int_{X}\theta_{\psi}^{n}>0$. Using
\[
\psi\leq\frac{\varphi_{1}+\psi_{1}+\varphi_{2,\varepsilon}+\psi_{2,\varepsilon}}{5}\leq\psi+B,
\]
and $S_{j,\varepsilon}\leq C(5\theta+dd^{c}(\varphi_{1}+\varphi_{2,\varepsilon}+\psi_{2,\varepsilon}+\psi_{1}))^{n}$ we obtain, for any Borel set $E$,  that
\[
\int_{E}S_{j,\varepsilon}\leq C'f(\capa_{\theta}(E)),\
\forall\varepsilon\in(0,1),j=1,2.
\]
For each $j\in\{1,2\}$ we also have that $S_{j,\varepsilon}\to S_j$, $\theta_{\psi_{2,\varepsilon}}^{n}\to\theta_{\psi_{2}}^{n}$ as $\varepsilon\to0$ in the weak sense of measures (see \cite[Theorem 2.3]{DDL2}). These measures are uniformly dominated by $\capa_{\theta}$. Note also that  $\varphi_{1}-\varphi_{2,\varepsilon},\varphi_{1}-\varphi_{2},\psi_{2,\varepsilon}-\psi_{1},\psi_{2}-\psi_{1}$ are uniformly bounded, quasi-continuous. Moreover, $\psi_{2,\varepsilon}-\psi_{1}\to\psi_{2}-\psi_{1}$, and $\varphi_{1}-\varphi_{2,\varepsilon}\to\varphi_{1}-\varphi_{2}$ in capacity as $\varepsilon\to 0$.  It  thus follows from Theorem \ref{thm: convergence} that 
\[
\lim_{\varepsilon\to 0} \int_{X}(\varphi_{1}-\varphi_{2,\varepsilon})\left(\theta_{\psi_{1}}^{n}-\theta_{\psi_{2,\varepsilon}}^{n}\right) = \int_{X}(\varphi_{1}-\varphi_{2})\left(\theta_{\psi_{1}}^{n}-\theta_{\psi_{2}}^{n}\right)
\]
and 
\[
\lim_{\varepsilon\to 0} \int_{X}\left(\psi_{1}-\psi_{2,\varepsilon}\right)(S_{1,\varepsilon}-S_{2,\varepsilon}) = \int_{X}\left(\psi_{1}-\psi_{2}\right)(S_{1}-S_{2}),
\]
finishing the proof of Step 1.2.

{\bf Step 2.} We merely assume that $\{\theta\}$ is big.  We can assume that $\theta+\omega$ is a K\"ahler form. For $s>2$ we apply the first step for $\theta_{s}:=\theta+s\omega$, which is also K\"ahler, to get 
\[
\int_{X}u\left((\theta_{s}+dd^{c}\psi_{1})^{n}-(\theta_{s}+dd^{c}\psi_{2})^{n}\right)=\int_{X}v
T_{s},
\]
where $u=\varphi_{1}-\varphi_{2}$, $v=\psi_{1}-\psi_{2}$ and 
\begin{flalign*}
T_{s}= & \sum_{k=0}^{n-1} (\theta_s +dd^c
\varphi_1)\wedge(\theta_s+dd^{c}\psi_{1})^{k}\wedge(\theta_{s}+dd^{c}\psi_{2})^{n-k-1}\\
- & \sum_{k=0}^{n-1}(\theta_s+dd^c
\varphi_2)\wedge(\theta_{s}+dd^{c}\psi_{1})^{k}\wedge(\theta_{s}+dd^{c}\psi_{2})^{n-k-1}.
\end{flalign*}
We thus obtain an equality between two polynomials in $s$. Identifying the coefficients we arrive at the conclusion.
\end{proof}

\subsection*{Proof of Theorem \ref{thm: integration by parts}}
We first assume that $\theta$ is K\"ahler, $u=\varphi_1-\varphi_2$ and $v=\psi_1-\psi_2$ where $\psi_1,\psi_2,\varphi_1,\varphi_2$ are $\theta$-psh. Fix $\phi\in\PSH(X,\theta)$ and for each $s\in[0,1]$, $j=1,2$, we set $\psi_{j,s}:=s\psi_{j}+(1-s)\phi$. Note that $\psi_{1,s}\simeq \psi_{2,s}$. It follows from Lemma \ref{lem: pre integration by parts} that for any $s\in[0,1]$,
\[
\int_{X}u\left(\theta_{s\psi_{1}+(1-s)\phi}^{n}-\theta_{s\psi_{2}+(1-s)\phi}^{n}\right)=\int_{X}sv T_s,
\]
where $T_s:= \sum_{k=0}^{n-1}\theta_{\varphi_1} \wedge \theta_{\psi_{1,s}}^k \wedge \theta_{\psi_{2,s}}^{n-k-1}-\sum_{k=0}^{n-1}\theta_{\varphi_2} \wedge \theta_{\psi_{1,s}}^k \wedge \theta_{\psi_{2,s}}^{n-k-1}$. 
We thus have an identity between two polynomials in $s$. Taking the first derivative in $s=0$ we obtain 
\[
\int_{X}udd^{c}v\wedge\theta_{\phi}^{n-1}=\int_{X}vdd^{c}u\wedge\theta_{\phi}^{n-1}.
\]
For the general case we can write $u=\varphi_1-\varphi_2$ and $v=\psi_1-\psi_2$ where $\psi_1,\psi_2,\varphi_1,\varphi_2$ are $A\omega$-psh, for some $A>0$ large enough. We apply the first step with $\theta$ replaced by $\theta + t\omega$, for $t>A$ to get
\[
\int_{X}u\;dd^c v\wedge(t\omega+\theta_
{\phi})^{n-1}=\int_{X}vdd^{c}u\wedge(t\omega+\theta_{\phi})^{n-1}.
\]
Identifying the coefficients of these two polynomials in $t$ we obtain
\[
\int_{X}udd^{c}v\wedge\theta_{\phi}^{n-1}=\int_{X}vdd^{c}u\wedge\theta_{\phi}^{n-1}.
\] 
We now consider $\theta=s_2\theta^2+....+s_n \theta^n$, $\phi:=s_{2}\phi_{2}+...+s_{n}\phi_{n}$ with $s_{2},...,s_{n}\in[0,1]$ and $\sum{s_{j}}=1$. We obtain an identity between two polynomials in $(s_2,...,s_n)$ and identifying the coefficients we arrive at the result.

\bibliographystyle{/Users/lu/Dropbox/Bib/amsplain_nodash}
\bibliography{/Users/lu/Dropbox/Bib/Biblio}

\end{document}